\documentclass[a4paper, 10pt, twoside, reqno, openright]{amsart}
\usepackage{amsmath}
\usepackage{amssymb}
\usepackage{amsthm}
\usepackage{amscd}
\usepackage{amsopn}
\usepackage{paralist}
\usepackage[UKenglish]{babel}
\usepackage[T1]{fontenc}
\usepackage{newtxtext,newtxmath}
\usepackage{multicol}
\usepackage{pifont}
\usepackage[top=1.0in, bottom=1.2in, left=1.15in, right=1.15in]{geometry}
\usepackage[cal=boondox, scr=euler, bb=ams]{mathalfa}
\usepackage{tikz}
\usepackage{wrapfig}
\usepackage{etex}
\usepackage{pictex}
\usepackage[all,cmtip]{xy}

\newtheoremstyle{Teorema}{5pt}{5pt}{\it}{}{\bf}{.}{ }{}
\theoremstyle{Teorema}
\newtheorem{Theorem}{Theorem}[section]
\newtheorem{Corollary}[Theorem]{Corollary}
\newtheorem{Proposition}[Theorem]{Proposition}

\newtheorem{Definition}[Theorem]{Definition}

\newtheorem{Lemma}[Theorem]{Lemma}
\newtheoremstyle{Annotazione}{5pt}{5pt}{\rm}{}{\it}{.}{ }{}
\theoremstyle{Annotazione}
\newtheorem{Remark}[Theorem]{Remark}

\def\SL{\mathrm{SL}}

\def\GL{\mathrm{GL}}

\def\CC{\mathbb{C}}
\def\bba{\mathbb{A}}
\def\bbz{\mathbb{Z}}

\def\bbc{\mathbb{C}}

\def\bbp{\mathbb{P}}

\def\Hom{\operatorname{Hom}}

\def\Aut{\operatorname{Aut}}

\def\Spec{\operatorname{Spec}}

\def\an{\mathrm{an}}

\makeatletter
\newcommand*\rel@kern[1]{\kern#1\dimexpr\macc@kerna}
\newcommand*\widebar[1]{%
  \begingroup
  \def\mathaccent##1##2{%
    \rel@kern{0.8}%
    \overline{\rel@kern{-0.8}\macc@nucleus\rel@kern{0.2}}%
    \rel@kern{-0.2}%
  }%
  \macc@depth\@ne
  \let\math@bgroup\@empty \let\math@egroup\macc@set@skewchar
  \mathsurround\z@ \frozen@everymath{\mathgroup\macc@group\relax}%
  \macc@set@skewchar\relax
  \let\mathaccentV\macc@nested@a
  \macc@nested@a\relax111{#1}%
  \endgroup
}

\makeatother

\usepackage{color}

\begin{document}

\title[Borel hyperbolicity]{Algebraicity of   analytic maps to a hyperbolic variety}
\author{Ariyan Javanpeykar}
\address{Ariyan Javanpeykar \\
Institut f\"{u}r Mathematik\\
Johannes Gutenberg-Universit\"{a}t Mainz\\
Staudingerweg 9, 55099 Mainz\\
Germany.}
\email{peykar@uni-mainz.de}

\author{Robert A. Kucharczyk}
\address{Robert A. Kucharczyk \\ Bonn,
 Germany.}
\email{robert.a.kucharczyk@gmail.com}



\begin{abstract} Let $X$ be an algebraic variety over $\CC$. We say that $X$ is Borel hyperbolic if, for every finite type reduced scheme $S$ over $\CC$, every holomorphic map $S^{\an}\to X^{\an}$ is algebraic. We use a transcendental specialization technique to prove that $X$ is Borel hyperbolic if and only if, for every smooth affine curve $C$ over $\CC$, every holomorphic map $C^{\an}\to X^{\an}$ is algebraic. We use the latter result to prove that Borel hyperbolicity shares many common features with other notions of hyperbolicity such as Kobayashi hyperbolicity.
\end{abstract}

\maketitle

\thispagestyle{empty}

\section{Introduction}  
In this paper we  study the   algebraicity of holomorphic maps into  a fixed variety.   If $X$ is a locally finite type scheme over $\CC$, let $X^{\an}$ be the associated complex-analytic space \cite[Expos\'e~XII]{SGA1}. If  $X$ and $Y$ are finite type schemes over $\mathbb{C}$ and $\varphi\colon X^{\an}\to Y^{\an}$ is a holomorphic map, then we say that $\varphi$ is \emph{algebraic} if there is a morphism of $\mathbb{C}$-schemes $f\colon X\to Y$ such that $f^{\an}  =\varphi$.

\begin{Definition}\label{def:borel_hyp}
A finite type scheme over $\CC$ is \emph{Borel hyperbolic} if, for every finite type reduced scheme $S$ over $\CC$,  any holomorphic map $S^{\an}\to X^{\an}$ is algebraic.
\end{Definition}

To motivate this specific terminology, we first explain why we choose to refer to varieties $X$ with the above property as  ``Borel hyperbolic''. For instance, 
the relations between Borel hyperbolicity and ``hyperbolicity'' in the usual sense are captured by the following theorem due to Brody,   Kobayashi, and Kwack.   Note that we recall the basic definitions from  hyperbolic geometry in Section \ref{section:proper_case}.

\begin{Theorem}[Why hyperbolic?]\label{thm:why_hyp} The following statements hold.
\begin{enumerate}
\item  If $X$ is a Borel finite type scheme over $\CC$, then $X$ is Brody hyperbolic. 
\item If $X$ is a one-dimensional finite type separated scheme, then $X$ is Brody hyperbolic if and only if $X$ is Borel hyperbolic.
\item If $X$ is a proper scheme over $\CC$, then $X$ is Brody hyperbolic if and only if $X$ is Borel hyperbolic.  
\item If there is a proper scheme $Y$ over $\CC$ and an open immersion $X\subset Y$  such that $X^{\an}$ is hyperbolically embedded in $Y^{\an}$, then $X$ is Borel hyperbolic.  
\end{enumerate} 
\end{Theorem}

 If $X$ is a finite type separated scheme over $\CC$, a useful string of implications that follows from  Theorem \ref{thm:why_hyp} is  the following:
 
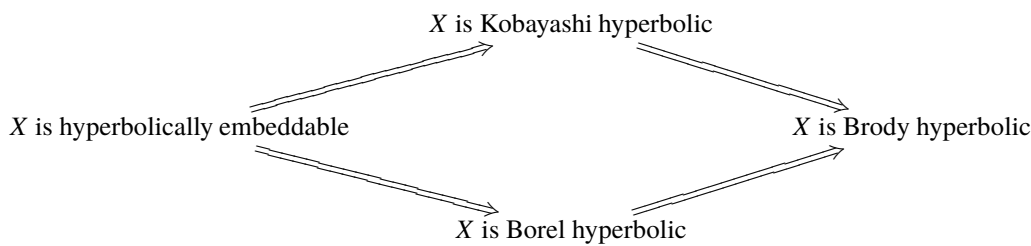
\[
 \xymatrix{   &  \textrm{$X$ is Kobayashi hyperbolic} \ar@{=>}[dr]   & \\  \textrm{$X$ is hyperbolically embeddable} \ar@{=>}[ur] \ar@{=>}[dr] &  & \textrm{$X$ is Brody hyperbolic}  \\ & \textrm{$X$ is Borel hyperbolic} \ar@{=>}[ur] &   }
 \]

    The reason we refer to such varieties as \emph{Borel} hyperbolic is because of Borel's theorem on locally symmetric varieties. More precisely, let $X$ be a finite type scheme  such that $X^{\an}$ is a locally symmetric variety, i.e.,  the universal cover of $X^{\an}$ is isomorphic to a bounded symmetric domain $D$ and $\pi_1(X^{\an})$ is an arithmetic (torsion-free) subgroup of $\mathrm{Aut}(D)$.  Then $X$ is a smooth quasi-projective  scheme over $\mathbb{C}$. It follows from a theorem of Borel that $X$ is in fact Borel hyperbolic; see  \cite[Theorem~3.10]{Borel1972} or  \cite[Theorem~5.1]{DeligneK3}.
    
   The aforementioned fact that  a locally symmetric variety is Borel hyperbolic is much harder to prove than the Brody hyperbolicity of $X$.  Indeed, as $\CC$ is a simply connected  topological space, the Brody hyperbolicity of a locally symmetric variety $X$ follows from Liouville's theorem on bounded holomorphic functions, as any holomorphic map  $\CC\to X^{\an}$   factors over a bounded domain.

\subsection{Motivation}   We were first led to investigate Borel hyperbolic varieties by   conjectures  of Green--Griffiths, Lang, and Vojta; see \cite{GrGr, Lang1, Vojta87}. These conjectures predict that several ``notions of hyperbolicity'' are equivalent, and thus should have the same formal properties.  By Theorem \ref{thm:why_hyp}, Borel hyperbolic varieties naturally fit into this conjectural framework. The results we establish in this paper verify predictions made by the aforementioned conjectures; see Theorems \ref{thm2} and \ref{et}.

Another reason for us to study Borel hyperbolicity (as a notion on its own) comes from arithmetic geometry. 
For example, the fact that the  fine moduli space of principally polarized abelian varieties over $\CC$ with level $3$ structure is Borel hyperbolic was first used by Deligne in his proof of the Weil conjectures for K3 surfaces \cite{DeligneK3}. Subsequently, Andr\'e used it in a similar fashion to prove the   Shafarevich conjecture and Tate conjecture for polarized K3 surfaces \cite{Andre}. Borel's theorem is also the starting point for studying the Kuga-Satake construction for polarized K3 surfaces, and its applications to Tate's conjecture \cite{Charles, Madapusi}.  
In light of these results, it seems reasonable to suspect that the Borel hyperbolicity of  ``period domains'', as conjectured by Griffiths, will play a similar important role in arithmetic geometry.

Indeed, in his seminal work on period maps and period domains \cite{GriffithsVarnI, GriffithsVarnIII}, Griffiths conjectured that the image of a ``period map'' is algebraizable and that the (a priori only holomorphic) period map is algebraic; see \cite{GGLR}. Part of Griffiths's conjecture can be formulated as saying that an algebraic variety which admits a quasi-finite period map is Borel hyperbolic. 

Furthermore, the complex algebraic stack of smooth  canonically polarized varieties    is Brody hyperbolic and even Kobayashi hyperbolic; see \cite{Schumacher, ToYeung, VZ}. In light of the aforementioned conjectures, it seems reasonable to suspect that this stack is also Borel hyperbolic.

 \subsection{Basic properties of Borel hyperbolic varieties}
 Motivated by conjectures of Green--Griffiths, Lang, and Vojta,  we investigate several basic properties of Borel hyperbolic varieties. For instance, we first show that Borel hyperbolicity persists over quasi-finite morphisms. Our precise result reads as follows, and should be considered as the ``Borel hyperbolic'' analogue of the similar statement  for Kobayashi hyperbolic varieties \cite[Theorem.~1]{Kwack} (see also\cite{Kobayashi}).
 
 \begin{Theorem}\label{thm2} Let $f\colon X\to Y $ be a quasi-finite morphism of finite type  separated schemes over $\CC$.
 If $Y$ is Borel hyperbolic, then $X$ is Borel hyperbolic.
 \end{Theorem}

 Next, we investigate how Borel hyperbolicity behaves with respect to finite \'etale maps. Any notion of ``hyperbolicity'' should descend along a finite \'etale map (see for instance \cite[Theorem.~3.2.8]{Kobayashi}). Our next result verifies this ``descent'' property  for the notion  Borel hyperbolicity.

 \begin{Theorem}\label{et} Let $f\colon X\to Y$ be a finite \'etale morphism of finite type separated schemes over $\CC$. Then $X$ is Borel hyperbolic if and only if $Y$ is Borel hyperbolic. 
 \end{Theorem}

The main technical result of this paper is the fact that    one can test Borel hyperbolicity on maps from smooth affine algebraic curves.  Our results reads as
 follows. 
\begin{Theorem}[Testing Borel hyperbolicity on maps from  curves]\label{thm:test_on_curves}
Let $X$ be a finite type separated   scheme over $\CC$. Then the following are equivalent.
\begin{enumerate}
\item The finite type separated scheme $X$ is Borel hyperbolic.
\item For every smooth affine connected curve $C$ over $\CC$, every holomorphic map $C^{\an}\to X^{\an}$ is algebraic.
\end{enumerate}
\end{Theorem}

\subsection{Outline of paper} 
Borel hyperbolicity is a ``transcendental'' property that     only makes sense for algebraic varieties. It therefore comes as no surprise that our study of Borel hyperbolic varieties   relies on several basic properties  of complex analytic spaces; we collect the relevant results in Section \ref{section22}. The main result of Section \ref{section2} is a transcendental specialization lemma for power series stated and proven in Section \ref{section:spec}; see Lemma \ref{Lem:TranscendentalSpecialisationPreservesTranscendentalFunctions}.

In Section \ref{section:bh}, we prove that one can test Borel hyperbolicity on holomorphic maps from smooth curves (Theorem \ref{thm:test_on_curves_in_body}). This theorem is an application of the main result of Section \ref{section:spec} and basic properties of complex analytic spaces.

We then use that one can test Borel hyperbolicity on maps from curves to prove Theorem \ref{thm:why_hyp}. Next, we use Riemann's existence theorem to prove that one can descend Borel hyperbolicity along finite \'etale maps. We then combine the fact that one can test Borel hyperbolicity on maps from curves with  a simple Lemma (Lemma \ref{lem:1_Borel_after_qf}) to conclude the proofs of Theorems  \ref{thm2} and   \ref{et}. 
 
 \subsection{Acknowledgements}
 We  thank    Benjamin Bakker, Damian Brotbek, Yohan Brunebarbe, Fr\'ed\'eric Campana, Lionel Darondeau,  Daniel Litt, and Erwan Rousseau for helpful discussions.
 This research was supported through the programme  ``Oberwolfach Leibniz Fellows'' by the Mathematisches Forschungsinstitut Oberwolfach in 2018.  The first named author gratefully acknowledges  support from SFB Transregio/45. This work was begun when the second named author was an ETH Fellow.

 \section{Transcendental specialization}\label{section2}
 
\subsection{Recollections about analytification}\label{section21}

With the hope towards greater clarity, structure sheaves of schemes (in particular sheaves of regular functions on varieties) will be denoted by $\mathcal{O}$, and structure sheaves of complex-analytic spaces (in particular sheaves of holomorphic functions on reduced complex-analytic spaces) will be denoted by~$\mathcal{H}$. This is also the convention used in Serre's GAGA \cite{Serre}.  The standard reference for the   theory of complex-analytic spaces is \cite{GrauertRemmert}.

Recall that there is an ``analytification'' functor $X\mapsto X^{\an}$ from the category of finite type schemes over $\CC$ to the category of complex-analytic spaces, see \cite[Expos\'e~XII]{SGA1}. It can be described in various ways, including the following universal property: for every finite type scheme $X$ over $\CC$ there is a complex-analytic space $X^{\an }$ together with a morphism of locally ringed spaces over $\Spec\CC$:
\begin{equation*}
i_X\colon (X^{\an },\mathcal{H}_{X^{\an }})\to (X,\mathcal{O}_X),
\end{equation*}
such that for every complex-analytic space $\mathfrak{Y}$, every morphism of locally ringed spaces $(\mathfrak{Y},\mathcal{H}_{\mathfrak{Y}})\to (X,\mathcal{O}_X)$ factors uniquely over~$i_X$.

For a finite type  $\CC$-scheme $X$ the map $i_X^{\sharp}\colon\mathcal{O}(X)\to\mathcal{H}(X^{\an })$ is injective. We will therefore often identify a regular function on $X$ with the holomorphic function on $X^{\an }$ it defines. Indeed,  often we will consider a holomorphic function $f$ on $X^{\an }$, and then say ``$f$ is regular'', meaning ``$f$ is the image under $i_X^{\sharp }$ of a regular function on~$X$''.

Recall that, for a finite type scheme $X$ over $\bbc$, the obvious map $X^{\an }\to X(\bbc )$ is a bijection. Moreover for every point $x\in X(\bbc )$, we obtain a local homomorphism of local rings $\mathcal{O}_{X,x}\to\mathcal{H}_{X^{\an },x}$; this is always injective but in general not surjective. The induced homomorphism on completions $\hat{\mathcal{O}}_{X,x}\to\hat{\mathcal{H}}_{X^{\an },x}$, however, is always an isomorphism  \cite[Th\'eor\`eme~XII.1]{SGA1}. For a holomorphic function $f$ on an open neighbourhood of $x$ in $X^{\an }$ we therefore obtain an element of $\hat{\mathcal{O}}_{X,x}$ which we call the \emph{Taylor expansion} of $f$ at~$x$.

The universal property of the map $(X^{\an },\mathcal{H}_{X^{\an }})\to (X,\mathcal{O}_X)$ turns analytification into a functor: for a morphism $f\colon X\to Y$ of schemes of finite type over $\bbc$ we obtain a morphism of complex-analytic spaces $f^{\an }\colon X^{\an }\to Y^{\an }$, uniquely characterized by the condition that the diagram
\begin{equation*}
\xymatrix{
	X^{\an }\ar[r]^-{f^{\an }} \ar[d]_-{i_X} & Y^{\an }\ar[d]^-{i_Y}\\
	X\ar[r]_-f & Y
	}
\end{equation*}
commutes. The thus defined map $\Hom (X,Y)\to\Hom (X^{\an },Y^{\an })$ is injective. We therefore will, as in the case of functions, often identify a morphism $X\to Y$ with its analytification $X^{\an }\to Y^{\an }$, and say that a morphism $X^{\an }\to Y^{\an }$ is algebraic, or algebraizes, if it comes from a (necessarily unique) morphism $X\to Y$.

   \subsection{Complex-analytic results}\label{section22}
  We next assemble some structure results about ring extensions occurring in complex-analytic geometry.
For the first auxiliary result, recall that a holomorphic map between complex-analytic spaces is called finite if it is  proper with finite fibres, and that a complex-analytic space $\mathfrak{X}$ is called irreducible if it cannot be written as a union $\mathfrak{X}=\mathfrak{X}_1\cup\mathfrak{X}_2$ with $\mathfrak{X}_i\subsetneqq\mathfrak{X}$ closed complex subspaces.
\begin{Lemma}\label{Lem:FiniteHolomorphicMapWithInverseIsIsomorphism}
	If a finite holomorphic map $\pi\colon\mathfrak{Y}\to\mathfrak{X}$ of reduced and irreducible complex-analytic spaces admits a holomorphic section $s\colon\mathfrak{X}\to\mathfrak{Y}$ then it is an isomorphism.
\end{Lemma}
\begin{proof}
	Since $\pi$ is proper and $\mathfrak{X}$ and $\mathfrak{Y}$ are locally compact, $s$ is also proper. Since it is injective, it is then also finite. Hence $s(\mathfrak{X})$ is a closed analytic subspace of~$\mathfrak{X}$.
		
	By Sard's Theorem there is an open dense subspace $\mathfrak{U}\subseteq\mathfrak{X}$ such that $\pi^{-1}(\mathfrak{U})\to\mathfrak{U}$ is a finite covering space. Hence $s(\mathfrak{U})$ is an open subspace of $\pi^{-1}(\mathfrak{U})$. Therefore we can find a decomposition $\mathfrak{Y}= s(\mathfrak{X})\cup (\mathfrak{Y}\setminus s(\mathfrak{U}))$ into closed complex subspaces. Now as $\mathfrak{Y}$ is irreducible and $s(\mathfrak{U})\neq\varnothing$, hence $\mathfrak{Y}\setminus s(\mathfrak{U})\neq\mathfrak{Y}$, we conclude that $s(\mathfrak{X})=\mathfrak{Y}$.
	
	Therefore $\pi$ and $s$ are mutually inverse bijections. Both of them are finite holomorphic maps, hence they are isomorphisms of complex-analytic spaces.
\end{proof}

We will also need an auxiliary construction. Suppose that $\mathfrak{X}$ is a complex-analytic space and $B\subseteq\mathcal{H}(\mathfrak{X})$ is a $\CC$-subalgebra. Then, by the universal property of affine schemes within the category of locally ringed spaces (see \cite[Errata et Addenda, Corollaire~1.8.3]{EGAII}), there exists a unique morphism of locally ringed spaces over~$\Spec\CC$:
\begin{equation*}
\varphi\colon\mathfrak{X}\to\Spec B
\end{equation*}
which on global sections induces the inclusion $B\hookrightarrow\mathcal{H}(\mathfrak{X})$. Now, if $B$ is a finitely generated $\CC$-algebra, then $Y=\Spec B$ is a $\CC$-scheme of finite type, and then by the universal property of analytifications, $\varphi$ induces a morphism of complex-analytic spaces
\begin{equation}\label{eqn:MapFromAnalyticSpaceToSpectrumSubalgebra}
\varphi\colon\mathfrak{X}\to Y^{\an }.
\end{equation}
By unravelling the definitions, we see that on underlying points this is the map $\mathfrak{X}\to Y^{\an }\cong\Hom_{\CC }(B,\CC)$ sending $x\in\mathfrak{X}$ to the evaluation morphism $B\to\CC$, $f\mapsto f(x)$.

\begin{Proposition}\label{Prop:RegularFunctionsIntegrallyClosedInHolomorphicFunctions}
If  $X$ is a finite type integral scheme over $\CC$, then the ring $\mathcal{O}(X)$ of regular functions is integrally closed in the ring $\mathcal{H}(X^{\an })$ of holomorphic functions.
\end{Proposition}
\begin{proof}
	It is straightforward to reduce to the case where $X$ is affine. Write $X=\Spec A$, so that $\mathcal{O}(X)=A$.
	Let $f\in\mathcal{H}(X^{\an })$ be integral over $A$ and set $B=A[f]$; we will show that $A=B$. By assumption, $B$ is a finite extension of $A$, and since it is contained in the domain $\mathcal{H}(X^{\an })$ it must be a domain itself. Hence $Y=\Spec B$ is also a complex variety, coming with a finite surjective morphism $\pi\colon Y\to X$ induced by the inclusion $A\hookrightarrow B$. We need to show that $\pi$ is an isomorphism, for then $f\in A$.
	
	There is a canonical \emph{analytic} section $s\colon X^{\an }\to Y^{\an }$, constructed from the inclusion $B\hookrightarrow\mathcal{H}(X^{\an })$ as in~(\ref{eqn:MapFromAnalyticSpaceToSpectrumSubalgebra}). Hence, by Lemma~\ref{Lem:FiniteHolomorphicMapWithInverseIsIsomorphism} the holomorphic map $\pi^{\an }\colon Y^{\an}\to X^{\an}$ is an isomorphism of complex-analytic spaces. Therefore, by \cite[Proposition~XII.3.1]{SGA1}, the morphism $\pi$  is an isomorphism.
\end{proof}
\begin{Proposition}\label{Prop:HolomorphicFunctionsAreIntegrallyClosedInOpen}
	Let $\mathfrak{X}$ be a normal complex-analytic space; let $A\subset\mathfrak{X}$ be a proper closed analytic subset. Then the ring of holomorphic functions $\mathcal{H}(\mathfrak{X})$ is integrally closed in $\mathcal{H}(\mathfrak{X}\setminus A)$.
\end{Proposition}

\begin{proof}
	Assume $f$ is holomorphic on $\mathfrak{X}\setminus A$ and it satisfies a relation of the form $f^d+a_{d-1}f^{d-1}+\dotsb +a_0=0$, where the $a_i$ are holomorphic on all of~$\mathfrak{X}$. Then around each point of $A$ the $a_i$ are bounded, hence so is $f$, and by Riemann's extension theorem (see  \cite[Satz~13]{GrauertRemmert1958}) it can be extended to a holomorphic function on all of~$\mathfrak{X}$.
\end{proof}

\begin{Remark} Note that   we need to assume that $\mathfrak{X}$ is normal in Proposition \ref{Prop:HolomorphicFunctionsAreIntegrallyClosedInOpen}: for instance, if $\mathfrak{X}$ is a non-normal Stein space and $A\subset\mathfrak{X}$ contains the locus of non-normality, then   $\mathcal{H}(\mathfrak{X})$ is not integrally closed in $\mathcal{H}(\mathfrak{X}\setminus A)$.
\end{Remark}

 \begin{Corollary}\label{Cor:RegularFunctionsIntegrallyClosedInHolomorphicFunctionsOnOpen}
	Let $X$ be a normal integral finite type scheme over $\CC$    and let $A\subset X^{\an }$ be a proper closed analytic subset. Then the ring of regular functions $\mathcal{O}(X)$ is integrally closed in the ring of holomorphic functions $\mathcal{H}(X^{\an }\setminus A)$. 
\end{Corollary}
\begin{proof}
	By Proposition  \ref{Prop:RegularFunctionsIntegrallyClosedInHolomorphicFunctions}, the ring $\mathcal O(X)$ is integrally closed in $\mathcal H(X^{\an})$. Moreover, by Proposition \ref{Prop:HolomorphicFunctionsAreIntegrallyClosedInOpen}, the ring $\mathcal H(X^{\an})$ is integrally closed in $\mathcal{H}(X^{\an}\setminus A)$. Thus, we conclude that $\mathcal{O}(X)$ is integrally closed in $\mathcal{H}(X^{\an}\setminus A)$.
\end{proof}

 \subsection{A specialization lemma for power series}\label{section:spec}
 Our aim in this section is to prove a ``transcendental'' specialization lemma. Similar (but weaker) results are proven in \cite{Osgood}.
 
 In the following we will assume that $k\subset\bbc$ is an algebraically closed subfield such that $\bbc$ has infinite transcendence degree over~$k$. Then we can embed the polynomial ring $k[x_1,\ldots ,x_{n+1}]$ into a polynomial ring of one dimension lower, $\bbc [z_1,\ldots ,z_n]$, as follows: choose some $\lambda_1,\ldots ,\lambda_n\in\bbc$ which are algebraically independent over $k$; this is possible by our assumption. Then we define a ring homomorphism
\begin{equation}\label{eqn:TranscendentalSpecialisationOnPolynomials}
\iota =\iota_{\lambda_1,\ldots ,\lambda_n}\colon k[x_1,\ldots ,x_{n+1}]\to\bbc [z_1,\ldots ,z_n]
\end{equation}
by letting $\iota |_k$ be the inclusion $k\hookrightarrow\bbc$, sending $x_j$ to $z_j$ for $1\le j\le n$, and sending $x_{n+1}$ to the linear polynomial $\lambda_1z_1+\dotsb +\lambda_nz_n$. This homomorphism extends in an obvious way to rings of formal power series:
\begin{equation}\label{eqn:TranscendentalSpecialisationOnPowerSeries}
\iota\colon k[\! [x_1,\ldots ,x_{n+1}]\! ]\to\bbc [\! [z_1,\ldots ,z_n]\!] .
\end{equation}
\begin{Lemma}\label{Lem:TranscendentalSpecialisationIsInjective}
	Under the given assumptions the maps (\ref{eqn:TranscendentalSpecialisationOnPolynomials}) and (\ref{eqn:TranscendentalSpecialisationOnPowerSeries}) are injective.
\end{Lemma}
\begin{proof}
	It suffices to show injectivity for~(\ref{eqn:TranscendentalSpecialisationOnPowerSeries}). If
	\begin{equation*}
	f=\sum_{i_1,\ldots ,i_{n+1}\ge 0}a_{i_1,\ldots ,i_{n+1}}x_1^{i_1}\dotsm x_{n+1}^{i_{n+1}}\in k[\! [x_1,\ldots ,x_{n+1}]\! ]
	\end{equation*}
	then
	\begin{equation*}
	\begin{split}
	\iota (f)&=\sum_{i_1,\ldots ,i_{n+1}\ge 0}a_{i_1,\ldots ,i_{n+1}}z_1^{i_1}\dotsm z_n^{i_n}(\lambda_1z_1+\dotsb +\lambda_nz_n)^{i_{n+1}}\\
	&=\sum_{i_1,\ldots ,i_{n+1}\ge 0}a_{i_1,\ldots ,i_{n+1}}z_1^{i_1}\dotsm z_n^{i_n}\sum_{\ell_1+\dotsb +\ell_n={i_{n+1}}}{i_{n+1}\choose \ell_1,\ldots ,\ell_n}\lambda_1^{\ell_1}\dotsm\lambda_n^{\ell_n}z_1^{\ell_1}\dotsm z_n^{\ell_n}\\
	&=\sum_{\substack{i_1,\ldots ,i_n\ge 0\\ \ell_1,\ldots ,\ell_n\ge 0}}{\ell_1+\dotsb +\ell_n\choose \ell_1,\ldots ,\ell_n}a_{i_1,\ldots ,i_n,\ell_1+\dotsb +\ell_n}\lambda_1^{\ell_1}\dotsm\lambda_n^{\ell_n}z_1^{i_1+\ell_1}\dotsm z_n^{i_n+\ell_n}\\
	&=\sum_{s_1,\ldots ,s_n\ge 0}\bigg(\sum_{\substack{i_1+\ell_1=s_1\\ \cdots \\ i_n+\ell_n=s_n}}{\ell_1+\dotsb +\ell_n\choose \ell_1,\ldots ,\ell_n}a_{i_1,\ldots ,i_n,\ell_1+\dotsb +\ell_n}\lambda_1^{\ell_1}\dotsm\lambda_n^{\ell_n}\bigg) z_1^{s_1}\dotsm z_n^{s_n}.
	\end{split}
	\end{equation*}
	Thus, if $\iota (f)=0$, then all the coefficients
	\begin{equation*}
	\sum_{\substack{i_1+\ell_1=s_1\\ \cdots \\ i_n+\ell_n=s_n}}{\ell_1+\dotsb +\ell_n\choose \ell_1,\ldots ,\ell_n}a_{i_1,\ldots ,i_n,\ell_1+\dotsb +\ell_n}\lambda_1^{\ell_1}\dotsm\lambda_n^{\ell_n}
	\end{equation*}
	have to be zero; since the multinomial coefficients ${\ell_1+\ldots +\ell_n\choose \ell_1,\ldots ,\ell_n}$ are positive integers, the $a_{i_1,\ldots ,i_{n+1}}$ are elements of $k$ and the $\lambda_i$ are algebraically independent over $k$, this can only happen if $a_{i_1,\ldots ,i_n,\ell_1+\dotsb +\ell_n}=0$ for all choices of $i_j,\ell_j\ge 0$. But then $f=0$.
\end{proof}
\begin{Lemma}\label{Lem:TranscendentalSpecialisationPreservesTranscendentalFunctions}
Let  $g\in k[\! [x_1,\ldots ,x_{n+1}]\! ]$. If  $\iota (g)\in\bbc[\! [z_1,\ldots ,z_n]\!]$ is an algebraic function (i.e. when interpreted as an element of the quotient field $\bbc (\! (z_1,\ldots ,z_n)\! )$ it is algebraic over the subfield $\bbc (z_1,\ldots ,z_n)$),
	then $g$ is   an algebraic function (i.e.\ $g$ is an element of $k(\! (x_1,\ldots ,x_{n+1})\! )$ algebraic over $k(x_1,\ldots ,x_{n+1})$).
\end{Lemma}
\begin{proof}
	By assumption, there exist some $d\ge 1$ and rational functions $a_0,\ldots ,a_{d-1}\in\bbc (z_1,\ldots ,z_n)$ such that
	\begin{equation}\label{eqn:MinimalPolynomialIotaOfG}
	\iota (g)^d+a_{d-1}\iota (g)^{d-1}+\dotsb +a_0=0;
	\end{equation}
	we may assume that this is the unique such equation with minimal~$d$. Write $K=k(\lambda_1,\ldots ,\lambda_n)\subset\bbc$; we claim that the $a_i$ are already contained in $K(z_1,\ldots ,z_n)$. Indeed, let the group $\Aut (\bbc /K)$ of all field automorphisms of $\bbc$ which fix $K$ pointwise operate on $\bbc (\! (z_1,\ldots ,z_n)\! )$ via its action on coefficients. Then $\iota (g)$ is fixed under $\Aut (\bbc /K)$, hence applying an element $\sigma\in\Aut (\bbc /K)$ to (\ref{eqn:MinimalPolynomialIotaOfG}) we find another equation
	\begin{equation}\label{eqn:MinimalPolynomialFieldAutConjugate}
	\iota (g)^d+\sigma (a_{d-1})\iota (g)^{d-1}+\dotsb +\sigma (a_0)=0.
	\end{equation}
	Then substracting (\ref{eqn:MinimalPolynomialFieldAutConjugate}) from (\ref{eqn:MinimalPolynomialIotaOfG}) we obtain an algebraic relation of degree at most $d-1$ for $\iota (g)$ over $\bbc (z_1,\ldots ,z_n)$, which by assumption is only possible if this relation is identically zero. Hence $\sigma (a_j)=a_j$ for all~$j$. Since this is assumed to hold for all $\sigma\in\Aut (\bbc /K)$, it implies that the $a_j$ lie in $K(z_1,\ldots ,z_n)$.
	
	We can rewrite this latter field as the image under the obvious extension of $\iota$ of the field
	\begin{equation*}
	k(x_1,\ldots ,x_n,x_{n+1},\lambda_1,\ldots ,\lambda_{n-1}),
	\end{equation*}
	where the $2n$ generating elements are algebraically independent over~$k$. Hence we may treat them as formal variables. We can then rewrite (\ref{eqn:MinimalPolynomialIotaOfG}) in the form
	\begin{equation*}
	g^d+a_{d-1}'g^d+\dotsb +a_0'=0
	\end{equation*}
	with some $a_j'\in k(x_1,\ldots ,x_{n+1},\lambda_1,\ldots ,\lambda_{n-1})$. Since $g\in k(x_1,\ldots ,x_{n+1})$, the same relation will hold true if we specialize each $\lambda_j$ for $1\le j\le n-1$ to a suitable element of~$k$. This is then the desired relation which shows $g$ to be algebraic over $k(x_1,\ldots ,x_{n+1})$.
\end{proof}

\subsection{Generic hyperplanes}\label{section:generic}
 The following definition will allow us to use the ``transcendental'' specialization lemma (Lemma \ref{Lem:TranscendentalSpecialisationPreservesTranscendentalFunctions}) in an algebro-geometric  context.  
\begin{Definition}
	Let $k\subset\bbc$ be an algebraically closed subfield, and let $n\ge 1$. An $n$-dimensional complex linear subspace $H\subset\bbc^{n+1}$ is called \emph{$k$-generic} if it can be defined by an equation
	\begin{equation}\label{eqn:DefiningKGenericHyperplane}
	H=\{ (z_1,\ldots ,z_{n+1})\in\bbc^{n
	+1}\mid z_{n+1}=\lambda_1z_1+\lambda_2z_2+\dotsb +\lambda_nz_n \}
	\end{equation}
	where $\lambda_1,\lambda_2,\ldots ,\lambda_n\in\bbc$ are algebraically independent over~$k$.
\end{Definition}
Clearly, if $k$ is countable (or, more generally, if the transcendence degree of $\bbc$ over $k$ greater than $n$) there exist uncountably many $k$-generic subspaces in~$\bbc^{n+1}$. Moreover, if $H$ is $k$-generic and $g\in\GL_{n+1}(k)$ then $g(H)$ is again $k$-generic; this shows that the notion makes sense for arbitrary finite-dimensional complex vector spaces with $k$-structure.

For a $k$-generic subspace $H\subset\bbc^{n+1}$ given by the equation (\ref{eqn:DefiningKGenericHyperplane}) the functions $z_1,\ldots ,z_n$ serve as natural $\bbc$-linear coordinates on $H$. The obvious map $\mathcal{O}(\bba_k^{n+1})\to\mathcal{O}(H)$ which sends each polynomial in $n+1$ variables over~$k$ to its restriction to $H$ is then precisely the map $\iota$ constructed in (\ref{eqn:TranscendentalSpecialisationOnPolynomials}), and similarly for formal functions in (\ref{eqn:TranscendentalSpecialisationOnPowerSeries}). Note that Lemma~\ref{Lem:TranscendentalSpecialisationIsInjective} can   be interpreted as saying that a polynomial over $k$ is uniquely determined by its restriction to a $k$-generic subspace, and similarly for formal functions. Similarly,  Lemma~\ref{Lem:TranscendentalSpecialisationPreservesTranscendentalFunctions} can be interpreted as saying that the algebraicity of a formal power series over $k$   can be checked at its restriction to a $k$-generic subspace.

\section{Borel hyperbolicity}\label{section:bh}  Recall that a finite type scheme $X$ over $\CC$ is Borel hyperbolic if, for all reduced finite type schemes $Y$ over $\mathbb C$, every  holomorphic map  $\varphi\colon Y^{\an}\to  X^{\an}$ algebraizes (Definition \ref{def:borel_hyp}).  
As we discussed in the introduction,  this notion of ``hyperbolicity'' is motivated by  Borel's theorem on locally symmetric varieties \cite[Theorem~3.10]{Borel1972}.
 
In this section we collect some basic properties of Borel hyperbolic varieties. We start with  its relation to Brody hyperbolicity. Recall that a finite type scheme $X$ over $\CC$ is \emph{Brody hyperbolic} if every holomorphic map $\mathbb{C}\to X^{\an}$ is constant.

 \begin{Lemma}\label{Lemma:Borel_implies_Brody0}
Let $X$ be a finite type scheme over $\CC$. If every analytic map $\CC\to X^{\an}$ algebraizes, then $X$ is Brody hyperbolic.
 \end{Lemma}
 \begin{proof}
Let $\varphi\colon\CC =\mathbb A_{\mathbb C}^{1,\an} \to X^{\an}$ be a morphism of complex-analytic spaces. Note that both $\varphi$  and $\varphi\circ \exp$  algebraize, where $\exp\colon\mathbb C\to \mathbb C$ is the exponential map. This implies that $\varphi$ is constant: otherwise $\varphi\circ \exp$ would be an algebraic map with countably infinite fibres, which is impossible.
 \end{proof}
 
 \begin{Lemma}\label{Lemma:Borel_implies_Brody}
 Let $X$ be a   finite type   scheme over $\CC$. If $X$ is Borel hyperbolic, then $X$ is Brody hyperbolic.
 \end{Lemma}
 \begin{proof}
This follows from the definition of Borel hyperbolicity and Lemma \ref{Lemma:Borel_implies_Brody0}.
 \end{proof}

  \begin{Lemma}\label{Lemma:Borel_via_red}
  	Let $X$ be a finite type  scheme    over $\CC$ and let $X_{\textup{red}}$ be the associated reduced closed subscheme. If $X_{\textup{red}}$ is Borel hyperbolic, then $X$ is Borel hyperbolic.
  \end{Lemma}
  \begin{proof}
  	Let $Y$ be a reduced finite type scheme over $\CC$ and let $\varphi\colon Y^{\an}\to X^{\an}$ be a morphism. Since $Y$ is reduced, the morphism $\varphi$ factors (uniquely) via $X_{\textrm{red}}$. Since $X_{\textrm{red}}$ is Borel hyperbolic, the induced morphism $Y^{\an}\to X_{\textrm{red}}^{\an}$ algebraizes, so that $\varphi$ algebraizes.
  \end{proof}
 
 \begin{Remark}\label{remark:why_red}
 	Of course, Lemma~\ref{Lemma:Borel_via_red} is an artifact of our choice to only consider reduced test schemes in Definition~\ref{def:borel_hyp}. We make this restriction for a good reason. For instance, the curve $X = \mathbb{P}^1\setminus \{0,1,\infty\}$ is a locally symmetric variety: $X^{\an }$ is isomorphic to $\Gamma (2)\backslash\mathfrak{H}$, where $\mathfrak{H}$ is the complex upper half-plane and $\Gamma (2)$ is the principal congruence subgroup of level 2 in $\SL_2(\bbz )$. So $X$ should be (and it is) Borel hyperbolic. Still, if $\varepsilon$ is a formal variable with $\varepsilon^2=0$, then the morphism $f\colon\mathbb{A}^{1,\an}_{\mathbb{C}[\varepsilon]}\to X^{\an}$ defined by  $f(z) = 2+\exp(z)\varepsilon$ does not algebraize.  
   \end{Remark}

 \begin{Remark} Let $X$ be a Borel hyperbolic  finite type reduced scheme    over $\mathbb{C}$. Then $\mathrm{Aut}(X) = \mathrm{Aut}(X^{\an})$. In fact, for all finite type reduced schemes $Y$ over $\CC$, we have $\mathrm{Isom}(Y,X) = \mathrm{Isom}(Y^{\an},X^{\an})$. Consequently, the ``algebraic structure'' on $X^{\an}$ is unique. (It is therefore no coincidence that any pair of   non-isomorphic   algebraic varieties whose associated analytic spaces are isomorphic  are not hyperbolic; see for instance \cite[Chap.~6.3, p. 232-235]{Har70} and \cite[\S 7]{Neeman}.)
\end{Remark}

 \subsection{Testing  algebraicity on maps from curves}
 We now set out to show that in deciding whether a scheme is Borel hyperbolic one can assume that the test schemes are smooth affine curves, which will come very handy in some later proofs.
 \begin{Proposition}[Dimension lemma]\label{Prop:DimensionLemma}
	Let $V$ and $X$ be complex algebraic varieties, where $V$ is normal and has dimension at least two, and let $f\colon V^{\an }\to X^{\an }$ be a holomorphic map. Suppose that for every closed algebraic subvariety $H\subset V$ of codimension one, the  composition
	\begin{equation*}
	\tilde{H}^{\an }\overset{\nu^{\an }}{\to }H^{\an }\hookrightarrow V^{\an }\overset{f}{\to }X^{\an }
	\end{equation*}
	is an algebraic morphism $\tilde{H}\to X$, where $\nu\colon\tilde{H}\to H$ is the normalisation morphism. Then $f$ itself is algebraic, $f\colon V\to X$.
\end{Proposition}
 \begin{proof}
We proceed in seven steps.

\textsc{Step~1:} \emph{Shrinking $V$ and~$X$}.\ -- We may, and will, assume that $V$ is an affine variety and that $f(V^{\an })$ is Zariski-dense in~$X$.\footnote{
The proof would simplify if we could already assume that $f(V^{\an })$ is contained in an affine subvariety of~$X$. However, for arbitrary holomorphic maps between complex algebraic varieties this is not necessarily the case, as the uniformisation $p\colon (\bba^1)^{\an }\to E^{\an}$ of an elliptic curve $E$ shows (note that $p|_{V^{\an }}$ is still surjective for every nonempty open subvariety $V\subset\bba^1$). We therefore need to argue more carefully.}

Choose some affine dense open subvariety $U\subseteq X$ and a closed embedding $j\colon U\hookrightarrow\bba^m$, and set $A=f^{-1}(X^{\an }\setminus U^{\an })$. By the assumptions made above, $A$ is a proper closed analytic subset of~$V^{\an }$. Denote the resulting map
\begin{equation}\label{eqn:DefinitionOfTheHolomorphicFunctionsGi}
V^{\an}\setminus A\overset{f}{\to }U^{\an }\overset{j^{\an }}{\to }\bbc^m
\end{equation}
by $g$ and its components by $g_1,\ldots ,g_m\colon V^{\an }\setminus A\to\bbc$. The $g_i$ are holomorphic functions. (Note that they cannot be simultaneously analytically continued to any open set $D$ with $V^{\an}\setminus A\subset D\subseteq V^{\an }$, because otherwise by analytic continuation $f(D)$ would still be contained in $U^{\an }$, which contradicts the construction of~$A$. Of course, it is still possible that an individual $g_i$ can be analytically extended to some part of~$A$.)

\textsc{Step~2:} \emph{Noether normalisation}.\ -- Choose a finite surjective morphism $\pi\colon V\to\bba^{n+1}$, where $n+1=\dim V>1$. By generic smoothness, $\pi$ becomes \'etale over a dense Zariski-open subvariety of $\bba^{n+1}$; we may assume that $0$ is contained in this subvariety, and we choose some $\tilde{0}\in V(\bbc )$ with $\pi (\tilde{0})=0$. We may also assume that $\tilde{0}\notin A$, i.e.\ that the $g_i$ are defined and holomorphic at~$\tilde{0}$.

\textsc{Step~3:} \emph{The subfield~$k$}.\ -- We choose a countable algebraically closed subfield $k\subset\bbc$ such that `the entire situation is defined over~$k$'; more precisely, there is a model $\pi\colon V_k\to\bba^{n+1}_k$ of $\pi\colon V=V_{\bbc }\to\bba^{n+1}_{\bbc }$ over $k$ (note that then $\tilde{0}$ can be identified with a unique point in $V_k(k)$), and the Taylor series of the holomorphic functions $g_i$ at $\tilde{0}$ lie in the subring $\hat{\mathcal{O}}_{V_k,\tilde{0}}\subset\hat{\mathcal{O}}_{V_{\bbc },\tilde{0}}$.

\textsc{Step 4:} \emph{The $k$-generic hyperplane}.\ -- Since $k$ is countable, there exists a $k$-generic hyperplane $P\subset\bba^{n+1}_{\bbc }$ through $0$ (Section \ref{section:generic}), given by an equation (\ref{eqn:DefiningKGenericHyperplane}) with $\lambda_1,\ldots ,\lambda_n\in\bbc$ algebraically independent over~$k$. Here we view $P$ as a $\bbc$-scheme which comes with an isomorphism to $\bba^n_{\bbc }$, with coordinates $z_1,\ldots ,z_n$. The scheme morphism $P\to\bba^{n+1}_k$ induces a ring homomorphism on completed local rings $\hat{\mathcal{O}}_{\bba^n_k,0}\to\hat{\mathcal{O}}_{P,0}$; this is precisely the homomorphism $k[\! [x_1,\ldots ,x_{n+1}]\! ]\to\bbc [\! [z_1,\ldots ,z_n]\! ]$ given by  (\ref{eqn:TranscendentalSpecialisationOnPowerSeries}) in Section \ref{section:spec}. In particular, the homomorphism $\hat{\mathcal{O}}_{\bba^n_k,0}\to\hat{\mathcal{O}}_{P,0}$ is injective by Lemma~\ref{Lem:TranscendentalSpecialisationIsInjective}.

We let $H=H_{\bbc }\subset V_{\bbc }$ be the irreducible component of $\pi^{-1}(P)$ which contains $\tilde{0}$ (note that this is unique because $\pi$ is \'etale at $\tilde{0}$). Then $H$ is a closed subvariety of codimension one in $V_{\bbc }$, and we denote its normalisation by $\tilde{H}$. By our assumption, the restrictions of the holomorphic functions $g_i$ to $\tilde{H}^{\an }\setminus A$ extend to rational functions $h_i\colon \tilde{H}_{\bbc }\dashrightarrow\bbp^1_{\bbc }$ which are regular at~$\tilde{0}$.

The commutative diagram of pointed schemes
\begin{equation*}
\xymatrix{
	(V_k,\tilde{0}) \ar[d]_-{\pi } & (\tilde{H}_{\bbc },\tilde{0}) \ar[l] \ar[d]^-{\pi |_{\tilde{H}}} \\
	(\bba^{n+1}_k,0) & \ar[l] (P_{\bbc },0)
	}
\end{equation*}
gives rise to a commutative diagram of completed local rings
\begin{equation}\label{eqn:CDInProofOfDimensionLemma}
\xymatrix{
	\hat{\mathcal{O}}_{V_k,\tilde{0}} \ar[r] & \hat{\mathcal{O}}_{\tilde{H}_{\bbc },\tilde{0}}  \\
	k[\! [x_1,\ldots ,x_{n+1}]\! ] \ar[u]^-{\pi^{\ast }}_-{\cong } \ar[r]_-{\iota } & \bbc [\! [z_1,\ldots ,z_n]\! ] \ar[u]_-{\pi^{\ast }}^-{\cong };
}
\end{equation}
the vertical maps in (\ref{eqn:CDInProofOfDimensionLemma}) are isomorphisms because $\pi$ is \'etale at~$\tilde{0}$ and thus $H$ is normal at~$\tilde{0}$.

\textsc{Step 5:} \emph{Formal algebraization}.\ -- The Taylor expansions of the holomorphic functions $g_i$ at $\tilde{0}$ are elements of $\hat{\mathcal{O}}_{V_k,\tilde{0}}$, hence by (\ref{eqn:CDInProofOfDimensionLemma}) they can be viewed as elements of $k[\! [x_1,\ldots ,x_{n+1}]\! ]$. Their images $\iota (g_i)\in\bbc [\! [z_1,\ldots ,z_n]\! ]\cong\hat{\mathcal{O}}_{\tilde{H}_{\bbc },\tilde{0}}$ can be interpreted as the Taylor expansions of the rational functions $h_i$ at~$\tilde{0}$. Since the composition $\tilde{H}\overset{\nu}{\to }H\overset{\pi }{\to }P$ is a finite morphism, the $h_i\in\bbc [\! [z_1,\ldots ,z_n]\! ]$ are then algebraic over $\bbc (z_1,\ldots ,z_n)$. By Lemma~\ref{Lem:TranscendentalSpecialisationPreservesTranscendentalFunctions}, the $g_i$ are then algebraic over $k(x_1,\ldots ,x_{n+1})$.

To spell this out, there exist rational functions $a_{0,i},a_{1,i},\ldots ,a_{d_i-1,i}\in k(x_1,\ldots ,x_{n+1})$, which via $\pi\colon V\to\bba^{n+1}$ can also be interpreted as rational functions on~$V$, such that
\begin{equation}\label{eqn:AlgebraicEquationForTheGiOverRationalFunctions}
g_i^{d_i}+a_{d_i-1,i}g_i^{d_i-1}+\dotsb a_{1,i}g_i+a_{0,i}=0.
\end{equation}
There is then a dense open subvariety $W\subseteq V$ such that the $a_{i,j}$ are actually regular functions on~$W$. Then the $g_i$ define elements of $\mathcal{H}(W^{\an }\setminus A)$, and (\ref{eqn:AlgebraicEquationForTheGiOverRationalFunctions}) can be viewed as an equation in $\mathcal{H}(W^{\an }\setminus A)$, with the $a_{i,j}\in\mathcal{O}(W)$. From Corollary~\ref{Cor:RegularFunctionsIntegrallyClosedInHolomorphicFunctionsOnOpen} we deduce that the $g_i$ are already elements in $\mathcal{O}(W)$. Note that here we use the assumption that $V$ (and thus also $W$) is normal.

\textsc{Step 6:} \emph{Algebraization on an open subvariety of~$V$}.\ -- We have seen that the $g_i$ are regular functions on a dense open subvariety of $V$, hence they can be extended to rational functions $g_i\colon V\dashrightarrow\bbp^1$. By construction and analytic continuation, the restriction of these rational functions to $V^{\an }\setminus A$ must be equal to the components of the map constructed in (\ref{eqn:DefinitionOfTheHolomorphicFunctionsGi}). By the remarks about analytic continuation in Step~1, $A$ must be equal to the union of the poles of $g_i$, in particular it is a closed algebraic subset of~$V$, and $W$ as in Step~5 can be chosen in such a way that $W^{\an }=V^{\an }\setminus A$.

\textsc{Step 7:} \emph{Global algebraization}.\  -- Resuming the previous steps, we see that we have shown the following statement: for every affine open $U\subseteq X$ and every $v\in V(\bbc )$ there is some quasi-affine open $W\subseteq V$ with $v\in W(\bbc )$ and $f(W^{\an })\subseteq U^{\an}$, and such that $f|_{W^{\an }}\colon W^{\an }\to U^{\an }$ is a regular map. This clearly shows that $f$ itself is a regular map.
\end{proof}

 \begin{Lemma}\label{Lemma:normalization_and_algebraization} Let $X$ be a finite type separated scheme over $\CC$.  
 Let $Y$ be a finite type reduced scheme and let $f\colon Y^{\an}\to X^{\an}$ be a holomorphic map. Let $\nu\colon\widetilde{Y}\to Y$ be the normalization of $Y$. If the composed morphism $\nu^{\an}\circ f$ algebraizes, then $f$ algebraizes.
 \end{Lemma}
\begin{proof}
	We claim that, for every Zariski-open $U\subseteq X$, the preimage $f^{-1}(U^{\an })$ is a Zariski-open subset of~$V$; this will allow us to assume that $X$ is affine. 
	
	Since $f\circ\nu^{\an}$ is a regular map, the preimage $(f\circ\nu^{\an} )^{-1}(U)$ is a Zariski-open algebraic subvariety of~$\tilde{V}$. Therefore $f^{-1}(U^{\an })=\nu^{\an} ((f\circ\nu^{\an} )^{-1}(U)) = \nu(f\circ\nu^{\an} )^{-1}(U))$ is a Zariski constructible subset of~$V$ (note that the normalisation morphism $\nu\colon\tilde{V}\to V$ is not necessarily open, so we cannot directly conclude that $f^{-1}(U^{\an })$ is Zariski-open). Since $f$ is continuous for the complex topology, $f^{-1}(U^{\an })$ is also open for the complex topology. Therefore, by \cite[ Corollaire~XII.2.3]{SGA1}, the subset $f^{-1}(U^{\an })$ is Zariski-open in ~$V$.
	
	So we may (and do) assume that $X$ is affine; by embedding it into some affine space we may assume moreover that $X=\bba^n$, and by considering the coordinate components of $f$ we may even further simplify to the case where $X=\bba^1$. Then $f$ is a holomorphic function on $V$ which becomes regular on~$\tilde{V}$. In particular it is integral over~$\mathcal{O}(V)$. By Proposition~\ref{Prop:RegularFunctionsIntegrallyClosedInHolomorphicFunctions} it is then regular on~$V$.
\end{proof}
\begin{Theorem}\label{thm:test_on_curves_in_body}
Let $X$ be a finite type separated scheme over~$\bbc$. Then the following are equivalent:
\begin{enumerate}
	\item $X$ is Borel hyperbolic.
	\item For every smooth complex algebraic curve $C$, every morphism of complex-analytic spaces $C^{\an }\to X^{\an }$ algebraizes.
\end{enumerate}
\end{Theorem}
\begin{proof}
	The implication (i)~$\Rightarrow$~(ii) is obvious. For the other direction, assume (ii), let $S$ be a reduced scheme of finite type over $\bbc$ and let $f\colon S^{\an }\to X^{\an }$ be a morphism of complex-analytic spaces. We need to show that $f$ algebraizes.
	
	First, by Lemma~\ref{Lemma:Borel_via_red} we can assume that $X$ is reduced. By considering irreducible components, we can also assume that both $X$ and $S$ are irreducible, hence they are varieties. By Lemma~\ref{Lemma:normalization_and_algebraization} we can also assume that $S$ is normal.
	
	If $S$ is a point, there is nothing to show. If $\dim S=1$, then $S$ is a smooth algebraic curve, so by assumption $f$ is algebraic. We therefore assume $\dim S\ge 2$. Consider, for every $1\le k\le\dim S$, the following statement: \textit{For every $k$-dimensional subvariety $H\subseteq S$, the composition}
	\begin{equation*}
	\tilde{H}^{\an }\overset{\nu^{\an }}{\to }H^{\an }\hookrightarrow S^{\an }\overset{f}{\to }X^{\an }
	\end{equation*}
	\textit{algebraizes.} Now, this statement is true for $k=1$ by assumption, and from Proposition~\ref{Prop:DimensionLemma} we conclude that if it holds for $k$ then it also holds for $k+1$. By induction, it then also holds for $k=n$, which means that $f$ itself is algebraic.
\end{proof}
\begin{Remark}
	One can further sharpen Theorem \ref{thm:test_on_curves_in_body}, for instance by only considering $C$ which are also affine, or only $C$ which are also hyperbolic. In both cases this follows from the fact that every smooth curve has an open cover by curves with the additional property. 
\end{Remark}

\subsection{Relating different analytic notions of hyperbolicity}\label{section:proper_case} 
In this section we gather known results relating the different notions of hyperbolicity. We start with an extension property for holomorphic maps. Write $\Delta =\{ z\in\bbc\mid\lvert z\rvert <1 \}$ and $\Delta^{\ast }=\Delta\setminus\{ 0\}$.

\begin{Definition}
A finite type  separated scheme $X$ over $\CC$ has the \emph{$\Delta^\ast$-extension property} if there is an open immersion $X\subset \overline{X}$ with $\overline{X}$ proper over $\CC$ such that, for every   morphism $f\colon\Delta^\ast \to X^{\an}$ there  is a morphism $\Delta\to \overline{X}^{\an}$ which extends $f$. 
\end{Definition}
 Note that a similar (but different) notion is studied in \cite{delta}. Moreover, Picard's Big Theorem can be stated as saying that $\mathbb{P}^1\setminus \{0,1,\infty\}$ has the $\Delta^\ast$-extension property. 
 
By using that one can test Borel hyperbolicity on maps from curves, we can prove that a variety $X$ having the $\Delta^\ast$-extension property is in fact Borel hyperbolic. 
 \begin{Corollary}\label{Corollary:Delta_implies_Borel}
 	Let $X$ be a finite type separated scheme over $\CC$. If $X$ has the $\Delta^\ast$-extension property, then $X$ is Borel hyperbolic.
 \end{Corollary}
 \begin{proof} Let $X\subset \overline{X}$ be as in the definition of the $\Delta^\ast$-extension property.
 	Let $C$ be a smooth curve and let $\varphi\colon C^{\an}\to X^{\an}$ be a holomorphic map. By   Theorem \ref{thm:test_on_curves_in_body}, it suffices to show that $\varphi$ is algebraic. Let $\overline{C}$ be the smooth compactification of $C$, and let $p\in \overline{C}\setminus C$ be a point. Let $\Delta \subset \overline{C}^{\an}$ be an open unit disk with origin $p$, and such that the punctured open unit disk $\Delta^\ast$ does not contain any point of $\overline{C}^{\an}\setminus C^{\an}$.
 	
Since $X$ has the $\Delta^\ast$-extension property, the induced holomorphic map $\Delta^\ast \to X^{\an}$ extends to a holomorphic map $\Delta \to \overline{X}^{\an}$.  Applying this to all $p$ in $\overline{C}\setminus C$, we see that the morphism $C^{\an}\to X^{\an}$ extends to a morphism $\overline{C}^{\an}\to \overline{X}^{\an}$. By the GAGA theorem for proper $\CC$-schemes   \cite[Corollaire~XII.4.5]{SGA1}, the morphism $\overline{C}^{\an}\to \overline{X}^{\an}$ algebraizes, so that $\varphi\colon C^{\an}\to X^{\an}$ algebraizes. This concludes the proof.
\end{proof}
 
 \begin{Remark}
By Corollary \ref{Corollary:Delta_implies_Borel} and Picard's Big Theorem, we see that $\mathbb{P}^1\setminus \{0,1,\infty\}$ is Borel hyperbolic. 
\end{Remark}

 If $X$ is a   finite type separated reduced scheme over $\CC$, then Kobayashi defined a pseudo-distance $d_X$ on $X^{\an}$; see \cite{Kobayashi}. We follow standard terminology and say that $X$ is \emph{Kobayashi hyperbolic} if $d_{X_{\textrm{red}}}$ is a distance function on $X_{\textrm{red}}$.
Suppose that $X$ is a dense open subscheme of a proper reduced scheme $Y$. We follow Kobayashi and say that $X$ is \emph{hyperbolically embedded} in $Y$ if $X$ is Kobayashi hyperbolic, and every point in $Y$ is hyperbolic \cite[\S 3.3]{Kobayashi}. We say that a finite type separated    scheme  $X$ over $\CC$ is \emph{hyperbolically embeddable} if there is a proper   scheme $Y$ and an open immersion $X\subset Y$ such that $X_{\textrm{red}}$ is hyperbolically embedded in $Y_{\textrm{red}}$.

 It is well-known that   Kobayashi hyperbolicity and Brody hyperbolicity are closely related. The following theorem makes this more precise, and also clarifies the relation to Borel hyperbolicity.

 \begin{Theorem}\label{thm:hyperbolicity_implications}
 	Let $X$ be a finite type integral   scheme over $\mathbb C$. The following statements hold.
 	\begin{enumerate}
 	 		\item If $X$ is separated and has the $\Delta^\ast$-extension property, then $X$ is Borel hyperbolic.
 		\item If $X$ is separated and hyperbolically embeddable, then $X$ is Kobayashi hyperbolic, Borel hyperbolic, and satisfies the $\Delta^\ast$-extension property.
 		\item If $X$ is Kobayashi hyperbolic, then $X$ is Brody hyperbolic.
 		\item If $X$ is Borel hyperbolic, then $X$ is Brody hyperbolic.
 	\end{enumerate}
 \end{Theorem}
 \begin{proof} Note that $(i)$ follows from  Corollary \ref{Corollary:Delta_implies_Borel}.  To prove $(ii)$, suppose that $X$ is separated and hyperbolically embedded in $\overline{X}$. It then follows  readily  that $X$ is Kobayashi hyperbolic. 
 Also, it follows from \cite[Theorem~6.3.7]{Kobayashi} that $X$ has the $\Delta^\ast$-extension property.  Thus, the Borel hyperbolicity of $X$ now follows from $(i)$.   Note that $(iii)$ and $(iv)$ follow from \cite[Proposition~3.6.1]{Kobayashi} and Lemma \ref{Lemma:Borel_implies_Brody}, respectively.
 \end{proof}

 \begin{proof}[Proof of Theorem \ref{thm:why_hyp}]  Note that Borel hyperbolic varieties are Brody hyperbolic (Lemma \ref{Lemma:Borel_implies_Brody}). This proves $(i)$.  Let $X$ be a proper scheme over $\mathbb C$. If $X$ is Brody hyperbolic, then $X$ is Kobayashi hyperbolic \cite[Theorem~3.6.3]{Kobayashi}, and thus (clearly) hyperbolically embedded in itself. Therefore, by   Theorem \ref{thm:hyperbolicity_implications},   the proper $\CC$-scheme $X$ is Borel hyperbolic.  This proves  $(iii)$. Moreover, $(iv)$ follows from   Theorem \ref{thm:hyperbolicity_implications}.(ii).
 
 Thus, to conclude the proof, it remains to prove  the statement about curves.
 Let $X$ be a finite type separated one-dimensional scheme over $\CC$. Assume $X$ is Brody hyperbolic. To show that $X$ is Borel hyperbolic, we may and do assume that $X$ is reduced (Lemma \ref{Lemma:Borel_via_red}). It is well-known that Brody hyperbolic curves are Kobayashi hyperbolic \cite{Kobayashi}. It is clear that every point in a compactification $\overline{X}$ of $X$ is hyperbolic (as it is isolated). We see that $X$ is hyperbolically embeddable, and thus Borel hyperbolic  (Theorem \ref{thm:hyperbolicity_implications}). This concludes the proof of $(ii)$.
 \end{proof}

\begin{Corollary}
Let $X$ be a proper scheme over $\mathbb{C}$. Then $X$ is Borel hyperbolic if and only if every morphism $\mathbb{A}^{1,\an}\to X^{\an}$ is algebraic.
\end{Corollary}
\begin{proof} If every holomorphic map $\mathbb{A}^{1,\an}\to X^{\an}$ algebraizes, then $X$ is Brody hyperbolic (Lemma~\ref{Lemma:Borel_implies_Brody0}). Conversely, if $X$ is Brody hyperbolic, then $X$ is Borel hyperbolic by  Theorem \ref{thm:why_hyp}.(\textit{iii}). In particular, if $X$ is Brody hyperbolic, then every morphism $\mathbb{A}^{1,\an}\to X^{\an}$ is algebraic.
\end{proof}

 \subsection{Borel hyperbolicity along quasi-finite maps}

 We first apply  Riemann's existence theorem to show that Borel hyperbolicity descends along \'etale coverings.
 \begin{Lemma}\label{Lemma:Et_for_Borel} Let $f\colon X\to Y$ be a   finite  \'etale   morphism of finite type schemes  over $\mathbb C$. If $X$ is Borel hyperbolic, then $Y$ is Borel hyperbolic. 
 \end{Lemma}
 \begin{proof}   Let $Z$ be a     finite type reduced scheme  and let $Z^{\an}\to Y^{\an}$ be a holomorphic map.   Consider the Cartesian diagram of complex-analytic spaces
 	\[\xymatrix{ \mathfrak{Z}' \ar[rr] \ar[d] & & X^{\an} \ar[d]^{\text{finite \'etale}} \\ Z^{\an} \ar[rr]  & &  Y^{\an}} \] Since $X^{\an}\to Y^{\an}$ is finite \'etale, it follows that $\mathfrak{Z}'\to Z^{\an}$ is finite \'etale. By Riemann's existence theorem \cite[Expos\'e XII, Th\'eor\`eme 5.1]{SGA1}, the morphism $\mathfrak{Z}'\to Z^{\an}$ algebraizes.  That is,  there exist a finite \'etale morphism $\tilde{Z}\to Z$ of schemes such that $\tilde{Z}^{\textrm{an}}$ is isomorphic to $\mathfrak{Z}'$ over $Z^{\textrm{an}}$.   Now, as $X$ is  a Borel hyperbolic finite type scheme, the morphism $\tilde{Z}^{\textrm{an}} \cong \mathfrak{Z}'\to X^{\an}$ algebraizes.  In particular, the composed morphism $$\tilde{Z}^{\an}\cong \mathfrak{Z}'\to X^{\an}\to Y^{\textrm{an}}$$ algebraizes. Therefore,  the morphism $Z^{\textrm{an}}\to Y^{\textrm{an}}$ algebraizes.
 	\end{proof}

To prove that Borel hyperbolicity persists along quasi-finite maps, we will use the following simple lemma.

\begin{Lemma}\label{lem:1_Borel_after_qf}
 			Let $f\colon X\to Y$ be a quasi-finite   morphism of finite type separated schemes over $\mathbb C$. Let $C$ be a finite type separated integral curve over $\CC$. Let $\varphi\colon C^{\an}\to X^{\an}$ be a holomorphic map. If the composition $C^{\an}\to Y^{\an}$ is algebraic, then $\varphi$ is algebraic. 
 			\end{Lemma}
 		
 		\begin{proof}  
  			The holomorphic map $C^{\textrm{an}}\to Y^{\textrm{an}}$ is the algebraization of a morphism $C\to Y$. 
 			If the map $C \to Y $ is constant, then the image of $C^{\textrm{an}} \to X^{\textrm{an}}$ is contained in a fibre of $f$. Since the fibres of $f$ are finite, the morphism    $C^{\textrm{an}}\to X^{\textrm{an}}$ is   constant, and thus   algebraic.  
 			
 			Thus, to prove the lemma, we may and do assume that $C\to Y$ is non-constant, hence quasi-finite. Let $D$ be the image of $C$ in $Y$. Note that $D$ is a one-dimensional integral  locally closed subscheme of $Y$.  Let $D^0$ be  a dense open Brody hyperbolic subvariety. Let $C^0$ be the inverse image of $D^0$ along $C\to D$.
  			
 		  Let $E = D^0\times_Y X$ be the pull-back of $f$ along $D^0\subset Y$.   Since $f$ is quasi-finite, the morphism $E\to D^0$ is quasi-finite. Since $D^0$ is Brody hyperbolic, we see that $E$ is a Brody hyperbolic   curve.  Now,  since $E$ is a Brody hyperbolic   curve, it is Borel  hyperbolic (Theorem \ref{thm:why_hyp}).
 		  
Note that $\varphi\colon C^{\textrm{an}}\to X^{\textrm{an}}$ factors via a holomorphic map $C^{\textrm{an}} \to E^{\textrm{an}}$. Since  $E$ is Borel hyperbolic, we conclude that $C^{\textrm{an}}\to E^{\textrm{an}}$ algebraizes. Since the inclusion    $E^{\an}\subset X^{\an}$ is algebraic, this proves that $C^{\textrm{an}}\to X^{\textrm{an}}$ algebraizes.
 		\end{proof}

  \begin{proof}[Proof of Theorem \ref{thm2}]
Let $C$ be a smooth quasi-projective curve over $\CC$ and let $C^{\an}\to X^{\an}$ be a holomorphic map. Since $Y$ is Borel hyperbolic, the composed holomorphic map $C^{\an}\to X^{\an}\to Y^{\an}$ is algebraic. Therefore, the holomorphic map $C^{\an}\to X^{\an}$ is algebraic (Lemma \ref{lem:1_Borel_after_qf}). Thus, we have shown that any holomorphic map from any smooth affine curve to $X$ algebraizes. It now follows from  Theorem \ref{thm:test_on_curves} that the finite type separated scheme $X$ is Borel hyperbolic.
  \end{proof}
 
 \begin{proof}[Proof of Theorem \ref{et}]
 Let $X\to Y$ be a finite \'etale morphism of finite type separated schemes over $\mathbb{C}$. If $X$ is Borel hyperbolic, then $Y$ is Borel hyperbolic (Lemma \ref{Lemma:Et_for_Borel} ). If $Y$ is Borel hyperbolic, then $X$ is Borel hyperbolic (Theorem \ref{thm2}).
 \end{proof}
 
 \begin{proof}[Proof of Theorem  \ref{thm:test_on_curves}]
 This is Theorem \ref{thm:test_on_curves_in_body}.
 \end{proof}

\bibliography{refs_borel}{}
\bibliographystyle{plain}

\end{document}